\documentclass[12pt]{article}
\title{Dp-finite fields V: topological fields of finite weight}
\author{Will Johnson}

\usepackage{amsmath, amssymb, amsthm}    	
\usepackage{fullpage} 	
\usepackage{amscd}
\usepackage{hyperref}
\usepackage[all]{xy}
\usepackage{centernot}

\DeclareMathOperator*{\forkindep}{\raise0.2ex\hbox{\ooalign{\hidewidth$\vert$\hidewidth\cr\raise-0.9ex\hbox{$\smile$}}}}

\newcommand{\crk}{\operatorname{c-rk}}
\newcommand{\wt}{\operatorname{wt}}

\newcommand{\Frac}{\operatorname{Frac}}

\newcommand{\val}{\operatorname{val}}

\newcommand{\dpr}{\operatorname{dp-rk}}

\newcommand{\Sub}{\operatorname{Sub}}
\newcommand{\Dir}{\operatorname{Dir}}

\newtheorem{theorem}{Theorem}[section] 
\newtheorem{lemma}[theorem]{Lemma}

\newtheorem{corollary}[theorem]{Corollary}
\newtheorem{fact}[theorem]{Fact}

\newtheorem{assumption}[theorem]{Assumption}
\newtheorem{conjecture}[theorem]{Conjecture}

\newtheorem{proposition}[theorem]{Proposition}
\newtheorem*{theorem-star}{Theorem}
\newtheorem*{conjecture-star}{Conjecture}

\theoremstyle{definition}
\newtheorem{definition}[theorem]{Definition}
\newtheorem{example}[theorem]{Example}

\theoremstyle{remark}
\newtheorem{remark}[theorem]{Remark}
\newtheorem{claim}[theorem]{Claim}

\newtheorem*{acknowledgment}{Acknowledgments}

\newcommand{\Qq}{\mathbb{Q}}

\newcommand{\Zz}{\mathbb{Z}}
\newcommand{\Kk}{\mathbb{K}}
\newcommand{\Nn}{\mathbb{N}}

\newcommand{\Oo}{\mathcal{O}}
\newcommand{\mm}{\mathfrak{m}}
\newcommand{\pp}{\mathfrak{p}}

\newenvironment{claimproof}[1][\proofname]
               {
                 \proof[#1]
                 
               }
               {
                 \endproof
               }

\begin{document}
\maketitle

\begin{abstract}
  We prove that unstable dp-finite fields admit definable
  V-topologies.  As a consequence, the henselianity conjecture for
  dp-finite fields implies the Shelah conjecture for dp-finite fields.
  This gives a conceptually simpler proof of the classification of
  dp-finite fields of positive characteristic.

  For $n \ge 1$, we define a local class of \emph{$W_n$-topological
    fields}, generalizing V-topological fields.  A $W_1$-topology is
  the same thing as a V-topology, and a $W_n$-topology is some
  higher-rank analogue.  If $K$ is an unstable dp-finite field, then
  the canonical topology of \cite{prdf} is a definable $W_n$-topology
  for $n = \dpr(K)$.  Every $W_n$-topology has between 1 and
  $n$ coarsenings that are V-topologies.  If the given $W_n$-topology
  is definable in some structure, then so are the V-topological
  coarsenings.
\end{abstract}

\section{Introduction}

There are two main conjectures on NIP fields.
\begin{conjecture}[Henselianity]
  Every NIP valued field $(K,v)$ is henselian.
\end{conjecture}
\begin{conjecture}[Shelah]
  Let $K$ be an NIP field.  Then one of the following holds:
  \begin{itemize}
  \item $K$ is finite
  \item $K$ is algebraically closed
  \item $K$ is real closed
  \item $K$ admits a non-trivial henselian valuation.
  \end{itemize}
\end{conjecture}
The Shelah conjecture is known to imply the henselianity conjecture
\cite{hhj-v-top}, as well as a full classification of NIP fields
\cite{NIP-char}.  For the Shelah conjecture, we may assume that $K$ is
sufficiently saturated.

These implications continue to hold in the restricted setting of
dp-finite fields (fields of finite dp-rank).  Our main result is a
partial converse to \cite{hhj-v-top}:
\begin{theorem}\label{thm1}
  The henselianity conjecture for dp-finite fields implies the Shelah
  conjecture for dp-finite fields.
\end{theorem}
For the special case of \emph{positive characteristic} dp-finite fields,
\begin{itemize}
\item The henselianity conjecture was proven by a very simple argument
  in \cite{prdf}, \S2.
\item The Shelah conjecture was proven by a very complicated argument
  in \cite{prdf}, \S10-11.
\end{itemize}
Using Theorem~\ref{thm1}, we get a conceptually simpler proof of the
second point.

If $K$ is a \emph{stable} dp-finite field, then $K$ is algebraically
closed or finite, by Proposition~7.2 in \cite{Palacin}.  The key to
proving Theorem~\ref{thm1} is
\begin{theorem}\label{thm2}
  Let $K$ be an \emph{unstable} dp-finite field.  Then $K$ admits a
  definable V-topology.
\end{theorem}
If $K$ is sufficiently saturated, then the definable V-topology in
Theorem~\ref{thm2} yields an externally definable valuation ring, by
Proposition~3.5 in \cite{hhj-v-top}.  The henselianity conjecture
applies to externally definable valuation rings, because of general
facts about Shelah expansions of NIP structures.  Therefore, the
henselianity conjecture implies the Shelah conjecture (for dp-finite
fields).

In order to prove Theorem~\ref{thm2}, we use the \emph{canonical
  topology} on $K$, defined in \cite{prdf}, Remark~6.18.
\begin{fact}
  Let $K$ be an unstable dp-finite field.
  \begin{itemize}
  \item The canonical topology is a field topology, i.e., the field
    operations are continuous.
  \item As $D$ ranges over definable subsets of $K$ with $\dpr(D) =
    \dpr(K)$, the sets $D - D$ range over a neighborhood basis of 0 in
    the canonical topology.
  \end{itemize}
  See (\cite{prdf2}, Corollaries~5.10 and 5.15) for a proof.
\end{fact}
We define a notion of a \emph{$W_n$-topology} on
a field.  These generalize V-topologies; in fact a $W_1$-topology is
the same thing as a V-topology.  Our main results on $W_n$-topologies
are the following:
\begin{theorem}~
  \begin{enumerate}
  \item If $\tau$ is a $W_n$-topology on a field $K$, then there is at
    least one V-topology $\tau'$ coarser than $\tau$.
  \item If $\tau$ is a $W_n$-topology on a field $K$, then the number
    of V-topological coarsenings is at most $n$.
  \item If $\tau$ is a \emph{definable} $W_n$-topology on some field
    $(K,+,\cdot,\ldots)$ (possibly with extra structure), then every
    V-topological coarsening is definable.
  \item Let $(K,+,\cdot,\ldots)$ be a field of dp-rank $n$ and let
    $\tau$ be the canonical topology.  Then $\tau$ is a definable
    $W_n$-topology.
  \end{enumerate}
\end{theorem}
In particular, the canonical topology is definable.  For rank 1, this
was proved in \cite{arxiv-myself}, and for rank 2 characteristic 0,
this was proved in \cite{prdf4}.

We also define a class of \emph{$W_n$-rings}, which generalize
valuation rings in the same way that $W_n$-topologies generalize
V-topologies.  In particular, a $W_1$-ring on $K$ is the same thing as
a valuation ring on $K$.  Any $W_n$-ring on $K$ induces a
$W_n$-topology.  Up to Prestel-Ziegler local equivalence, all
$W_n$-topologies arise from $W_n$-rings.

We avoid using the inflator machinery of \cite{prdf3} for the above
results.  Nevertheless, there seems to be a close connection between
inflators and $W_n$-rings, which we discuss in \S\ref{sec:inflators}.
In particular, the analysis of 2-inflators in \cite{prdf4} yields a
classification of $W_2$-topologies on fields of characteristic 0:
\begin{theorem}
  Let $K$ be a field of characteristic 0, and $\tau$ be a field
  topology.  Then $\tau$ is a $W_2$-topology if and only if $\tau$ is
  one of the following:
  \begin{itemize}
  \item a V-topology
  \item a ``DV-topology'' in the sense of \cite{prdf4}, Definition~8.18.
  \item a topology generated by two independent V-topologies.
  \end{itemize}
\end{theorem}
It may be possible to similarly classify $W_n$-topologies using
$n$-inflators.  In \S\ref{conjs:sec}, we discuss some conjectures
about $W_n$-rings and $W_n$-topologies.  One of these conjectures
would imply the Shelah conjecture (for dp-finite fields).

\subsection{Review of topological fields}
We will make heavy use of Prestel and Ziegler's machinery of
\emph{local sentences} and \emph{local equivalence} \cite{PZ}.  A
``ring topology'' on a field $K$ will mean a Hausdorff non-discrete
topology on $K$ such that the ring operations are continuous.  All
topologies we consider will be ring topologies on fields.  A ring
topology is determined by the set of neighborhoods of 0.  Following
\cite{PZ}, we identify a (ring) topology $\tau$ with its set of
neighborhoods of 0 (rather than its set of open sets).

One can consider a topological field $(K,\tau)$ as a two-sorted
structure with sorts $K$ and $\tau$.  A \emph{local sentence} is a
first-order sentence in this language, subject to the following
constraints on quantification over $\tau$:
\begin{itemize}
\item If there is universal quantification $\forall U \in \tau :
  \phi(U)$, then $U$ must occur positively in $\phi$.
\item If there is existential quantification $\exists U \in \tau :
  \phi(U)$, then $U$ must occur negatively in $\phi$.
\end{itemize}
These constraints ensure that one can replace the quantifiers
\[ \forall U \in \tau, \exists U \in \tau\]
with quantification over a neighborhood base.

For example,
\[ \forall U \in \tau ~ \exists V \in \tau : (1 + V)^{-1} \subseteq 1 + U\]
is a local sentence \footnote{More accurately, the local sentence is the following:
  \begin{equation*}
    \forall U \in \tau ~ \exists V \in \tau ~ \forall x \in V ~ \exists y \in U : (1 + x)(1 + y) = 1.
  \end{equation*}} expressing that the (ring) topology is a field topology.

Two topological fields are \emph{locally equivalent} if they satisfy
the same local sentences.  A topological field $(K,\tau)$ is
\emph{$\omega$-complete} if $\tau$ is closed under countable
intersections, i.e., an intersection of countably many neighborhoods
of 0 is a neighborhood of 0.  Every topological field is locally
equivalent to an $\omega$-complete topological field (\cite{PZ},
Theorem~1.1).

A subset $S \subseteq K$ is \emph{bounded} if the following
equivalent conditions hold:
\begin{itemize}
\item For every neighborhood $U \ni 0$, there is nonzero $c \in K$
  such that $c S \subseteq U$.
\item For every neighborhood $U \ni 0$, there is a neighborhood $V \ni
  0$ such that $S \cdot V \subseteq U$.
\end{itemize}
The equivalence is Lemma~2.1(d) in \cite{PZ}.

A ring topology is \emph{locally bounded} if there is a bounded
neighborhood $V$ of 0.  In this case, the set $\{c V : c \in
K^\times\}$ is a neighborhood basis of 0, by Lemma~2.1(e) in
\cite{PZ}.

If $R$ is a proper subring of $K$ and $K = \Frac(R)$, then $R$ induces
a locally bounded ring topology $\tau_R$ on $K$, for which either of
the following families are a neighborhood basis:
\begin{itemize}
\item The set of rescalings $c R$, where $c \in K^\times$.
\item The set of nonzero ideals of $R$.
\end{itemize}
Up to local equivalence, every locally bounded ring topology arises in
this way (\cite{PZ}, Theorem~2.2(a)).

We shall need the following variant of Lemma~1.4 and Theorem~2.2 in
\cite{PZ}:
\begin{proposition} \label{gap:prop}
  Let $(K,\tau)$ be a field with a Hausdorff non-discrete
  $\omega$-complete locally bounded ring topology.  Let $U$ be a
  bounded neighborhood of 0.  Let $R$ be the subring generated by $U$.
  Then $R$ is a bounded neighborhood of 0, $\Frac(R) = K$, and $\tau =
  \tau_R$.
\end{proposition}
\begin{proof}
  Recursively define $U_0 \subseteq U_1 \subseteq U_2 \subseteq \cdots$ by
  \begin{itemize}
  \item $U_0 = U \cup \{0,1\}$
  \item $U_{i+1} = U_i \cup (U_i - U_i) \cup (U_i \cdot U_i)$.
  \end{itemize}
  Then each $U_i$ is a bounded set, by Lemma~2.1 in \cite{PZ}.  By the
  comment at the start of \S 2 in \cite{PZ}, the union of the $U_i$'s
  is bounded.  This union is $R$.  Also, $R \supseteq U$, so $R$ is a
  neighborhood of 0.  Then $R$ is a bounded neighborhood, so the
  family of sets $\{c R : c \in K^\times\}$ is a neighborhood basis of
  0, by Lemma~2.1(e) in \cite{PZ}.  To see that $\Frac(R) = K$, note
  that for any $a \in K^\times$, the neighborhood $R \cap a R$ must
  strictly exceed $\{0\}$ (as the topology is non-discrete).  This
  implies that $a \in \Frac(R)$.
\end{proof}

\section{Rings of finite weight}
\subsection{Cube rank}
Let $R$ be a ring and $M$ be an $R$-module.  Let $\crk(M)$ or
$\crk_R(M)$ denote the \emph{cube rank} of $M$ as an $R$-module
(\cite{prdf3}, Definition~6.6).  Cube rank is an element of $\Nn \cup
\{\infty\}$ and can be characterized in two ways:
\begin{itemize}
\item By Remark~6.7 in \cite{prdf3}, $\crk(M) \ge n$ if and only if
  there are submodules $N' \le N \le M$ such that the subquotient
  $N/N'$ is isomorphic to a direct sum of $n$ non-trivial $R$-modules.
\item By Proposition~7.3 in \cite{prdf4}, $\crk(M) \ge n$ if and only
  if there are $m_1, \ldots, m_n \in M$ such that no $m_i$ is
  generated by the others:
  \begin{equation*}
    \forall i : m_i \notin R \cdot m_1 + \cdots + R \cdot m_{i-1} + R
    \cdot m_{i+1} + \cdots + R \cdot m_n.
  \end{equation*}
\end{itemize}
\begin{remark}
  Cube rank was called ``reduced rank'' in \cite{prdf,prdf3,prdf4},
  and is probably a well-known concept to people who study modules or
  lattices.  At a minimum, cube rank seems related to ``uniform
  dimension'' and ``hollow dimension'' in module theory.
\end{remark}
By Proposition~6.9 in \cite{prdf3}, cube rank has the following properties:
\begin{itemize}
\item $\crk(M) > 0$ if and only if $M$ is non-trivial.
\item $\crk(M \oplus N) = \crk(M) + \crk(N)$.
\item If $N$ is a submodule, quotient, or subquotient of $M$, then
  $\crk(N) \le \crk(M)$.
\item In a short exact sequence
  \begin{equation*}
    0 \to N \to M \to N' \to 0,
  \end{equation*}
  we have $\crk(M) \le \crk(N) + \crk(N')$.
\end{itemize}
Moreover, $\crk(-)$ is the smallest $\Nn \cup \{\infty\}$-valued
function with these properties (\cite{prdf3}, Proposition~E.2).

\begin{lemma}\label{ring-change:lem}
  If $M$ is an $R'$-module and $R \subseteq R'$, then $\crk_R(M) \ge
  \crk_{R'}(M)$.
\end{lemma}
\begin{proof}
  If $\crk_{R'}(M) \ge n$, witnessed by $m_1, \ldots, m_n$, then for
  any $i$,
  \begin{equation*}
    m_i \notin \sum_{j \ne i} R' \cdot m_j \supseteq \sum_{j \ne i} R
    \cdot m_j. \qedhere
  \end{equation*}
\end{proof}

\begin{remark}
  Occasionally, we will also need cube rank on modular lattices.  If
  $\Lambda$ is a modular lattice, then $\crk(\Lambda)$ can be
  characterized in one of several equivalent ways.
  \begin{itemize}
  \item $\crk(\Lambda) \ge n$ if there is a strict $n$-cube in $M$, in
    the sense of Definition~9.13 in \cite{prdf}.
  \item $\crk(\Lambda) \le n$ if for any $a_0, \ldots, a_n \in
    \Lambda$, there is $i$ such that
    \begin{equation*}
      a_0 \wedge \cdots \wedge a_n = a_0 \wedge \cdots \wedge
      \widehat{a_i} \wedge \cdots \wedge a_n.
    \end{equation*}
  \item $\crk(\Lambda) \le n$ if for any $a_0, \ldots, a_n \in
    \Lambda$, there is $i$ such that
    \begin{equation*}
      a_0 \vee \cdots \vee a_n = a_0 \vee \cdots \vee
      \widehat{a_i} \vee \cdots \vee a_n.
    \end{equation*}
  \end{itemize}
  These definitions are equivalent by Proposition~6.3 of \cite{prdf3}.  If $a, b$
  are two elements of $\Lambda$ with $a \ge b$, then
  $\crk_\Lambda(a/b)$ will denote the cube rank of the interval
  $[b,a]$, a sublattice of $\Lambda$.
\end{remark}

\subsection{$W_n$-rings}
Let $R$ be an integral domain with fraction field $K$.
\begin{lemma}\label{kr:lem}
  $\crk_R(R) = \crk_R(K)$.
\end{lemma}
\begin{proof}[Proof (cf. Lemma~10.25 in \cite{prdf4})]
  $\crk_R(R) \le \crk_R(K)$ because $R$ is a submodule of $K$.
  Conversely, suppose $\crk_R(K) \ge n$.  Then there are $m_1, \ldots,
  m_n \in K$ such that for any $i$,
  \begin{equation*}
    m_i \notin \sum_{j \ne i} R \cdot m_j.
  \end{equation*}
  Take non-zero $s \in R$ such that $s m_i \in R$ for all $i$.  Then the set $\{s
  m_1, \ldots, s m_n\}$ shows $\crk_R(R) \ge n$.
\end{proof}
\begin{definition}
  The \emph{weight} of $R$, written $\wt(R)$ is the value $\crk_R(R) = \crk_R(K)$.
  We say that $R$ is a \emph{$W_n$-ring (on $K$)} if $\wt(R) \le n$.
\end{definition}
\begin{proposition}
  $R$ is a $W_1$-ring if and only if $R$ is a
  valuation ring.
\end{proposition}
\begin{proof}
  By definition, $\crk_R(R) \le 1$ means that for any $x, y \in R$,
  \begin{equation*}
    x \in R \cdot y \text{ or } y \in R \cdot x.
  \end{equation*}
  This is the definition of a valuation ring.
\end{proof}

\begin{lemma}\label{super:lem}
  Let $R$ be a $W_n$-ring on $K$.  Let $L/K$ be
  a finite extension of degree $d$.  Let $R'$ be a subring of $L$
  containing $R$.  Then $R'$ is a $W_{dn}$-ring.

  In particular, if $R'$ is a subring of $K$ containing $R$, then $R'$
  is a $W_n$-ring.
\end{lemma}
\begin{proof}
  By Lemmas~\ref{ring-change:lem} and \ref{kr:lem},
  \begin{equation*}
    \crk_{R'}(R') \le \crk_{R'}(L) \le \crk_R(L) = \crk_R(K^d) = d \cdot \crk_R(K) \le d n. \qedhere
  \end{equation*}
\end{proof}
\begin{corollary}
  Let $R$ be a $W_n$-ring.  Let $\mm$ be a maximal ideal.  Then the
  localization $R_\mm$ is a $W_n$-ring.
\end{corollary}

\subsection{Maximal ideals and the integral closure}
\begin{proposition}
  Let $R$ be a $W_n$-ring.  Then $R$ has at most $n$ maximal ideals.
  In particular, $W_n$-rings are semilocal.
\end{proposition}
\begin{proof}
  If $\mm_1, \ldots, \mm_{n+1}$ are distinct maximal ideals of $R$,
  then
  \begin{equation*}
    R/(\mm_1 \cap \cdots \cap \mm_{n+1}) \cong (R/\mm_1) \times \cdots \times (R/\mm_{n+1})
  \end{equation*}
  by the Chinese remainder theorem, and so $\wt(R) = \crk_R(R) \ge n+1$.
\end{proof}
\begin{corollary}\label{jac:cor}
  If $R$ is a $W_n$-ring on $K$, and $R \ne K$, then the Jacobson
  radical of $R$ is non-trivial.
\end{corollary}

\begin{proposition}\label{multi:prop}
  Let $\Oo_1, \ldots, \Oo_n$ be pairwise incomparable valuation rings
  on a field $K$.  Then the intersection $\bigcap_{i = 1}^n \Oo_i$ is
  a ring of weight $n$.
\end{proposition}
\begin{proof}
  Lemma~6.5 in \cite{prdf3}.
\end{proof}
In the language of \cite{prdf2}, multivaluation rings have finite weight.

\begin{proposition}\label{iclose:prop}
  Let $R$ be a ring of weight $n$ on a field $K$.  Then
  the integral closure of $R$ (in $K$) is a multivaluation ring, an
  intersection of at most $n$ valuation rings on $K$.
\end{proposition}
\begin{proof}
  Let $\tilde{R}$ denote the integral closure.  Let $\mathcal{P}$ be
  the class of valuation rings on $K$ containing $R$.  On general
  grounds, $\tilde{R} = \bigcap \mathcal{P}$.  By
  Proposition~\ref{multi:prop} and Lemma~\ref{super:lem},
  $\mathcal{P}$ contains no antichains of size $n+1$.  By Dilworth's
  theorem, $\mathcal{P}$ is a union of $n$ chains $\mathcal{C}_1,
  \ldots, \mathcal{C}_n$.  Each intersection $\bigcap \mathcal{C}_i$
  is a valuation ring, and
  \begin{equation*}
    \bigcap \mathcal{P} = \bigcap_{i= 1}^n \left( \bigcap
    \mathcal{C}_i\right). \qedhere
  \end{equation*}
\end{proof}

\begin{remark}\label{nont:rem}
  If $R$ is non-trivial, i.e., $R \subseteq K$, then the integral
  closure is non-trivial.  To see this, take $\pp$ a maximal ideal of
  $R$.  Then $\pp \ne 0$.  Take nonzero $x \in \pp$.  An easy argument
  shows that $1/x$ is not integral over $R$.
\end{remark}

\section{W-topologies}

\subsection{Topologies from W-rings}
Let $R$ be a $W_n$-ring on a field $K$.  Suppose $R$ is
non-trivial, i.e., $R \ne K$.  By Example~1.2 in \cite{PZ}, $R$ induces a
Hausdorff non-discrete locally bounded ring topology on $K$, in which
the family
\begin{equation*}
  \{ c R : c \in K^\times\}
\end{equation*}
is a neighborhoods basis of 0.  Equivalently, the non-zero ideals of
$R$ are a neighborhood basis of 0.
\begin{proposition}
  The topology induced by $R$ is a field topology, i.e., division is
  continuous.
\end{proposition}
\begin{proof}
  This holds because the Jacobson radical of $R$ is non-trivial, as in
  the proof of Theorem~2.2(b) in \cite{PZ}.  In more detail, to verify that
  division is continuous it suffices to prove the following: for any
  non-zero ideal $I \le R$, there is a non-zero ideal $I' \le R$ such
  that
  \begin{equation*}
    (1 + I')^{-1} \le 1 + I.
  \end{equation*}
  Let $J$ be the Jacobson radical of $R$.  By Corollary~\ref{jac:cor},
  $J \ne 0$.  Let $I' = I \cap J$.  The intersection is non-zero as
  $R$ is a domain.  Then for any $x \in I'$, we have
  \begin{align*}
    x \in I' &\implies x \in J \implies 1 + x \in R^\times \implies
    \frac{-x}{1 + x} \in I' \\ & \implies \frac{-x}{1+x} \in I \implies
    \frac{1}{1+x} \in 1 + I. \qedhere
  \end{align*}
\end{proof}

\subsection{W-topologies}

\begin{definition}
  Let $K$ be a field.  A \emph{$W_n$-set} is a subset $S \subseteq K$
  such that for any $x_1, \ldots, x_{n+1} \in K$, there is an $i \le
  n+1$ such that
  \begin{equation*}
    x_i \in x_1 \cdot S + \cdots + x_{i-1} \cdot S + x_{i+1} \cdot S +
    \cdots + x_{n+1} \cdot S.
  \end{equation*}
\end{definition}
Note that an integral domain $R$ is a $W_n$-ring if and only if $R$ is
a $W_n$-set in $\Frac(R)$.
\begin{definition}
  A \emph{$W_n$-topology} on a field $K$ is a (Hausdorff non-discrete)
  locally bounded ring topology on $K$ such that for every
  neighborhood $U \ni 0$, there is $c \in K^\times$ such that $c \cdot
  U$ is a $W_n$-set.
\end{definition}

\begin{remark}\label{local-class:rem}
  The class of $W_n$-topologies is a \emph{local class}, cut out
  by finitely many \emph{local sentences} in the sense of \cite{PZ}.
  Specifically, it is cut out by the axioms of locally bounded
  non-trivial Hausdorff ring topologies plus the local sentence
  \begin{equation*}
    \forall U \in \tau ~ \exists c \ne 0 ~ \forall x_1, \ldots, x_{n+1} ~
    \bigvee_{i = 1}^n \left( x_i \in \sum_{j \ne i} c \cdot x_j \cdot
    U \right).
  \end{equation*}
\end{remark}
\begin{remark}\label{suff:rem}
  A locally bounded ring topology is a $W_n$-topology if and
  only if some bounded neighborhood $V \ni 0$ is a $W_n$-set.  Indeed,
  if such a $V$ exists, then for any neighborhood $U \ni 0$ there is
  $c \in K^\times$ such that $c U \supseteq V$, because $V$ is
  bounded.  Then $cU$ is a $W_n$-set because $V$ is a $W_n$-set.
  Conversely, suppose the topology is a $W_n$-topology.  Let $U$
  be a bounded neighborhood of 0.  Then there is $c \in K^\times$ such
  that $c U$ is a $W_n$-set.  But $cU$ is again a bounded neighborhood
  of 0.
\end{remark}
\begin{proposition}
  Let $R$ be a non-trivial $W_n$-ring on a field $K$.  Then the
  induced field topology is a $W_n$ topology.
\end{proposition}
\begin{proof}
  The bounded neighborhood $R$ is a $W_n$ set.
\end{proof}
\begin{lemma}\label{extend:lem}
  Let $K$ be a field with a $W_n$-topology.  Suppose that $K$ is
  $\omega$-complete as in \cite{PZ}.  Then the topology is induced by
  a $W_n$-ring $R$.  Moreover, given any fixed bounded set $S$, we may
  assume $S \subseteq R$.
\end{lemma}
\begin{proof}
  Let $U$ be a bounded neighborhood of 0.  Scaling $U$ by an
  element of $K^\times$, we may assume that $U$ is a $W_n$-set.  By
  Proposition~\ref{gap:prop}, $U \cup S$ is contained in a bounded
  subring $R \subseteq K$, and $R$ induces the topology.  Then $R$ is
  a $W_n$-ring (= $W_n$-set) on $K$, because it contains the $W_n$-set
  $U$.
\end{proof}
\begin{corollary}\label{basic:cor}
  Let $K$ be a field with a ring topology $\tau$.
  \begin{enumerate}
  \item $\tau$ is a $W_n$-topology if and only if $(K,\tau)$ is
    locally equivalent to a field with a topology induced by a
    $W_n$-ring.
  \item If $\tau$ is a $W_n$-topology, then $\tau$ is a field
    topology.
  \item If $\tau$ is a $W_n$-topology, then $\tau$ is a $W_m$ topology
    for $m > n$.
  \item $\tau$ is a $W_1$-topology if and only if $\tau$ is a
    V-topology.
  \end{enumerate}
\end{corollary}
\begin{definition}
  The \emph{weight} of a ring topology $\tau$ is the minimal $n$ such
  that $\tau$ is a $W_n$-topology, or $\infty$ if no such $n$ exists.
  We write the weight as $\wt(\tau)$.
\end{definition}
Then $\tau$ is a $W_n$-topology if and only if $\wt(\tau) \le n$.

\section{W-topologies on monster models}

\subsection{A criterion for definability}
Let $\Kk$ be some highly saturated field, possibly with extra
structure.  Recall that a topology is \emph{definable} if it has a
(uniformly) definable basis of opens.  In the case of a ring topology,
it suffices to produce a definable neighborhood basis of 0.  In the
case of a locally bounded ring topology, it suffices to produce a
definable bounded neighborhood $U$, as the definable family $\{ a U :
a \in \Kk^\times\}$ is then a definable neighborhood basis of 0.

As in \S6.4 of \cite{prdf4}, say that two subsets $X, Y \subseteq \Kk$
are \emph{co-embeddable} if there are $a, b \in \Kk^\times$ such that
$aX \subseteq Y$ and $bY \subseteq X$.  In a locally bounded ring
topology, the bounded neighborhoods of 0 form a single
co-embeddability class.
\begin{proposition}\label{upgrader:prop}
  Let $\tau$ be a $W_n$-topology on $\Kk$.  Suppose that there is $U
  \subseteq \Kk$ such that
  \begin{itemize}
  \item $U$ is a bounded neighborhood of 0 with respect to $\tau$.
  \item $U$ is $\vee$-definable or type-definable.
  \item $U$ is a subgroup of $(\Kk,+)$.
  \end{itemize}
  Then $U$ is co-embeddable with a definable set, and $\tau$ is a definable topology.
\end{proposition}
\begin{proof}
  Rescaling $U$, we may assume that $U$ is a $W_n$-set.  Take $m$
  minimal such that $U$ is a $W_m$-set.  Then there are $b_1, \ldots,
  b_m \in \Kk$ such that for any $i$,
  \begin{equation*}
    b_i \notin \sum_{j \ne i} b_j U.
  \end{equation*}
  \begin{claim} \label{cl0}
    For any $i$, $b_i$ is not in the closure of $\sum_{j \ne i} b_j U$.
  \end{claim}
  \begin{claimproof}
    In fact, $\sum_{j \ne i} b_j U$ is closed.  If $m = 1$, then
    $\sum_{j \ne i} b_j U = \{0\}$, which is closed because the
    topology is Hausdorff.  If $m > 1$, then $\sum_{j \ne i} b_j U$ is
    an additive subgroup of $\Kk$, and a neighborhood of 0.  Therefore
    it is a clopen subgroup.
  \end{claimproof}
  Let $S$ be the set of $x \in \Kk$ such that
  \begin{equation*}
    \bigvee_{i = 1}^m \left(b_i \in x U + \sum_{j \ne i} b_j U\right)
  \end{equation*}
  If $U$ is type-definable (resp. $\vee$-definable), then $S$ is
  type-definable (resp. $\vee$-definable), and $\Kk \setminus S$ is
  $\vee$-definable (resp. type-definable).
  \begin{claim} \label{cl1}
    If $x \notin S$, then $x \in \sum_{j = 1}^m b_j U$.
  \end{claim}
  \begin{claimproof}
    Otherwise, the set $\{b_1,\ldots,b_m,x\}$ witnesses that $U$ is
    not a $W_m$-set.
  \end{claimproof}
  \begin{claim} \label{cl2}
    There is a neighborhood $V \ni 0$ such that $V \cap S =
    \emptyset$.
  \end{claim}
  \begin{claimproof}
    By Claim~\ref{cl0}, each $b_i$ is not in the closure of $\sum_{j
      \ne i} b_j U$.  By Lemma~2.1(d) in \cite{PZ},
    \[ (b_i + V \cdot U) \cap \sum_{j \ne i} b_j U = \emptyset\]
    for small enough $V$.  We can choose $V$ to work across all $b_i$.
    Then if $x \in V \cap S$, there is some $i$ such that
    \begin{align*}
      b_i &\in xU + \sum_{j \ne i} b_j U \\
      \emptyset & \ne (b_i + xU) \cap \sum_{j \ne i} b_j U \subseteq (b_i + V \cdot U) \cap \sum_{j \ne i} b_j U,
    \end{align*}
    contradicting the choice of $V$.
  \end{claimproof}
  Claims~\ref{cl1} and \ref{cl2} imply that
  \begin{equation*}
    V \subseteq \Kk \setminus S \subseteq \sum_{j = 1}^m b_j U.
  \end{equation*}
  By Lemma~2.1 in \cite{prdf}, the right hand side is bounded.  Thus
  $\Kk \setminus S$ is a bounded neighborhood of 0.  So $\Kk \setminus
  S$ is co-embeddable with $U$.  One of $\{U,\Kk \setminus S\}$ is
  type-definable, and othe other is $\vee$-definable.  As in
  Remark~6.17 of \cite{prdf4}, this implies that $U$ is co-embeddable
  with a definable set $D$.  Then $D$ is a definable bounded
  neighborhood of 0, and $\tau$ is a definable topology.
\end{proof}
\begin{remark}
  It may be possible to drop the peculiar assumption that $U$ is an
  additive subgroup in Proposition~\ref{upgrader:prop}, but a more
  convoluted argument would be needed.
\end{remark}

\subsection{$\vee$-definable W-rings}
Let $\Kk$ be some highly saturated field, possibly with extra
structure.
\begin{proposition}\label{vee-to-def}
  Let $R$ be a non-trivial $\vee$-definable $W_n$-ring on
  $\Kk$.  Then $R$ is co-embeddable with a definable set $D$.
  Consequently, the field topology induced by $R$ is definable.
\end{proposition}
\begin{proof}
  Proposition~\ref{upgrader:prop} applied to the neighborhood $R$
  itself.
\end{proof}

\begin{proposition}\label{ic:prop}
  Let $R$ be a $\vee$-definable $W_n$-ring on $K$.  Let
  $\tilde{R}$ be the integral closure.  Then $\tilde{R}$ is
  $\vee$-definable.
\end{proposition}
\begin{proof}
  $\tilde{R}$ is the union of the $\vee$-definable sets
  \begin{equation*}
    S_n := \{x \in \Kk ~|~ \exists y_0, \ldots, y_n \in R : x^{n+1} =
    y_0 + y_1 x + \cdots + y_n x^n\}. \qedhere
  \end{equation*}
\end{proof}

\begin{proposition}\label{loc:prop}
  Let $R$ be a $\vee$-definable $W_n$-ring on a monster field
  $\Kk$.  Let $\pp$ be one of the maximal ideals of $R$.  Then the
  localization $R_\pp$ is also $\vee$-definable.
\end{proposition}
\begin{proof}
  Let $\pp_1, \ldots, \pp_n$ enumerate the maximal ideals of $R$, with
  $\pp_1 = \pp$.  For every subset $S \subseteq \{2, \ldots, n\}$, let
  $a_S$ be an element of $\pp_1$ such that
  \begin{equation*}
    a_S \in \pp_j \iff j \in S.
  \end{equation*}
  The $a_S$ exist by the Chinese remainder theorem.
  \begin{claim}
    For $x \in R$, the following are equivalent:
    \begin{itemize}
    \item $x \notin \pp_1$
    \item There is $S$ such that $1/(x + a_S) \in R$.
    \end{itemize}
  \end{claim}
  \begin{claimproof}
    If $x \in \pp_1$, then $x + a_S \in \pp_1$ for every $S$, so there
    is no $S$ with $x + a_S \in R^\times$.

    Conversely, if $x \notin \pp_1$, then we can find $a_S$ such that
    \begin{equation*}
      x \in \pp_i \iff a_S \notin \pp_i
    \end{equation*}
    for $i = 1, \ldots, n$.  Then $x + a_S \notin \pp_i$ for any $i$,
    so $x + a_S \in R^\times$.
  \end{claimproof}
  As $R$ is $\vee$-definable, it follows that $R \setminus \pp_1$ is
  $\vee$-definable.  Then
  \begin{equation*}
    R_{\pp} = \{ x/s : x \in R \text{ and } s \in R \setminus \pp_1\}
  \end{equation*}
  is $\vee$-definable as well.
\end{proof}

\subsection{V-topological coarsenings}
\begin{theorem}\label{defcourse}
  Let $(K,\tau)$ be a field with a $W_n$-topology.
  \begin{enumerate}
  \item \label{p1} There is at least one V-topological coarsening of $\tau$.
  \item \label{p2} There are at most $n$ such coarsenings.
  \item \label{p3} If $(K,\tau)$ is a definable topology (with respect
    to some structure on $K$), then every V-topological coarsening is
    definable.
  \end{enumerate}
\end{theorem}
\begin{proof}
  ~
  \begin{enumerate}
  \item Let $D \subseteq K$ be a bounded neighborhood of 0 that is a
    $W_n$-set.  Let $(K^*,D^*)$ be a highly saturated elementary
    extension of $(K,D)$.  Then $D^*$ defines a $W_n$-topology on
    $K^*$ that is $\omega$-complete.  Let $R$ be the $\vee$-definable
    subring generated by $D^*$.  By Proposition~\ref{gap:prop}, $R$ is
    a bounded neighborhood of 0, and a $W_n$-ring because it contains
    the $W_n$-set $D^*$.  So $R$ defines the same topology as $D^*$.
    Let $\tilde{R}$ be the integral closure of $R$, and let $\pp$ be a
    maximal ideal of $\tilde{R}$.  Then the localization
    $\tilde{R}_{\pp}$ is a $\vee$-definable valuation ring by
    Propositions~\ref{ic:prop}, \ref{loc:prop}.  Therefore it induces
    a definable V-topology on $(K^*,D^*)$.  This definable V-topology
    is coarser than the topology induced by $R$ or $D^*$, because
    $\tilde{R}_{\pp} \supseteq R$.  The statement ``there is a
    definable V-topology coarser than the topology induced by $D^*$''
    is expressed by a disjunction of first-order sentences, so it
    holds in the elementary substructure $(K,D)$.
  \item Let $\sigma_1, \ldots, \sigma_m$ be distinct V-topological
    coarsenings of $\tau$.  We claim $m \le n$.  By Remark~1.5 in
    \cite{PZ}, we may assume that all the topologies are
    $\omega$-complete.  (Local sentences can assert that $\tau$ is or
    isn't coarser than $\tau'$.)  Then $\tau$ is induced by a
    ring $R$ of weight at most $n$.  Additionally, $R$
    is bounded with respect to $\sigma_i$.  Therefore, there is a
    valuation ring $\Oo_i$ inducing $\sigma_i$ and containing $R$, by
    Lemma~\ref{extend:lem}.  As $i$ varies, the valuation rings
    $\Oo_i$ induce pairwise distinct topologies, so they must be
    pairwise incomparable.  As in the proof of
    Proposition~\ref{iclose:prop}, this implies $m \le n$.
  \item Let $\sigma$ be any V-topological coarsening of $\tau$.  Let
    $D$ and $B$ be bounded neighborhoods of 0 in $\tau$ and $\sigma$
    respectively, with $D$ definable in the given structure.  After
    rescaling, we may assume that $D$ is a $W_n$-set and $B$ is a
    $W_1$-set.  Because $\sigma$ is coarser than $\tau$, the set $B$
    is a neighborhood of 0 in $\tau$, and so there is $c$ such that $c
    D \subseteq B$.  Therefore $D$ is $\tau$-bounded.  Replacing $B$
    with $B \cup D$, we may assume $D \subseteq B$.

    Let $(K^*,D^*,B^*)$ be a highly saturated elementary extension.
    Let $R_D$ and $R_B$ be the subrings generated by $D^*$ and $B^*$.
    Then
    \begin{itemize}
    \item $R_D$ is a $W_n$-subring of $K^*$, $\vee$-definable in the
      reduct $(K^*,D^*$), and co-embeddable with $D^*$.
    \item $R_B$ is a $W_1$-subring (i.e., valuation ring) on $K^*$,
      $\vee$-definable in $(K^*,D^*,B^*)$, and co-embeddable with
      $B^*$.
    \item $R_B \supseteq R_D$.  Therefore $R_B$ contains the integral
      closure $\widetilde{R_D}$.  Writing this integral closure as an
      intersection of valuation rings $\Oo_1 \cap \cdots \cap \Oo_n$,
      there must be some $i$ such that $R_B \supseteq \Oo_i$, by
      Corollary~6.8 in \cite{prdf2}.
    \end{itemize}
    As in Part~\ref{p1}, the valuation ring $\Oo_i$ is
    $\vee$-definable in the reduct $(K^*,D^*)$.  By
    Proposition~\ref{vee-to-def}, there is some $C \subseteq K^*$
    definable in $(K^*,D^*)$, and co-embeddable with $\Oo_i$.  Then
    $C, \Oo_i, R_B, B^*$ are all co-embeddable.  (The inclusion $R_B
    \supseteq \Oo_i$ forces the two valuation rings to induce the same
    topology, hence to be co-embeddable.)

    The statement ``some set definable in $(K^*,D^*)$ is co-embeddable
    with $B^*$'' is expressed by a disjunction of first-order
    sentences, so it holds in the elementary substructure $(K,D,B)$.
    Therefore there is $C_0 \subseteq K$ definable in $(K,D)$ and
    co-embeddable with $B$.  So the V-topology $\sigma$ is definable
    in $(K,D)$. \qedhere
  \end{enumerate}
\end{proof}

While we are here, we note an analogue of Proposition~3.5 in
\cite{hhj-v-top}.
\begin{proposition}\label{hhj-like}
  Let $(\Kk,+,\cdot,\ldots)$ be a sufficiently saturated field,
  possibly with extra structure, and let $\tau$ be a definable
  $W_n$-topology on $\Kk$.  Then $\tau$ is induced by a
  $\vee$-definable, externally definable $W_n$-ring $R$ on $\Kk$.
\end{proposition}
\begin{proof}
  Let $D$ be a definable bounded neighborhood of 0.  Rescaling, we may
  assume $D$ is a $W_n$-set.  Passing to a reduct, we may assume the
  language is countable.  Let $K$ be a countable elementary
  substructure defining $D$.  Let $R$ be the union of all
  $K$-definable bounded sets.  Then $R$ is a $\vee$-definable subring,
  by Lemma~2.1(b-c) in \cite{PZ}.  Also, $R \supseteq D$, so $R$ is a
  $W_n$-ring.  By the remark at the start of \cite{PZ}, \S 2, a union
  of countably many bounded sets in $\Kk$ is still bounded.  Therefore
  $R$ is bounded.  Then $R$ is a bounded neighborhood of 0, so $R$
  induces the topology.
  \begin{claim}
    If $S_1, \ldots, S_n$ are $K$-definable bounded subsets, then
    there is $c \in K^\times$ such that $S_1 \cup \cdots \cup S_n
    \subseteq c D$.
  \end{claim}
  \begin{claimproof}
    There is $c \in \Kk^\times$ such that $S_1 \cup \cdots \cup S_n
    \subseteq c D$, by Lemma~2.1(e) in \cite{PZ}.  As $S_1, \ldots,
    S_n, D$ are $K$-definable, we can choose $c \in K^\times$.
  \end{claimproof}
  It follows that $R$ can be written as a directed union $\bigcup_{c
    \in K^\times} c D$.  Therefore $R$ is externally definable and
  $\vee$-definable.
\end{proof}

\section{Golden lattices}\label{sec:gold}
\begin{definition}\label{gold:def}
  Let $K$ be a field.  A \emph{golden lattice} on $K$ is a
  collection $\Lambda$ of subgroups of $(K,+)$, satisfying the
  following criteria:
  \begin{description}
  \item[(Lattice)] $\Lambda$ is a bounded sublattice of $\Sub_\Zz(K)$.  In other words
    \begin{itemize}
    \item $0, K \in \Lambda$
    \item If $G, H \in \Lambda$, then $G \cap H, G + H \in \Lambda$.
    \end{itemize}
  \item[(Scaling)] $\Lambda$ is closed under the action of $K^\times$: if $G \in \Lambda$ and $c \in K^\times$, then $c G \in \Lambda$.
  \item[(Rank)] $\Lambda$ has finite cube rank.
  \item[(Intersection)] Let $\Lambda^+ = \Lambda \setminus \{0\}$.  Then $\Lambda^+$
    is closed under finite intersections.
  \item[(Non-degeneracy)] $\Lambda$ is strictly bigger than $\{0,K\}$.
  \end{description}
\end{definition}
\begin{example}
  If $R$ is a $W_n$-ring on $K$, then $\Sub_R(K)$ is a golden lattice on $K$.
\end{example}
\begin{example}
  Proposition~10.1 in \cite{prdf} says that for certain dp-finite
  fields $\Kk \succeq K$, the lattice of type-definable $K$-linear
  subspaces of $\Kk$ is a golden lattice.
\end{example}
For the remainder of \S\ref{sec:gold}, we assume the following:
\begin{assumption}
  $\Lambda$ is a golden lattice of rank $r$ on a field $K$, and
  $\Lambda^+$ is the unbounded sublattice $\Lambda \setminus \{0\}$.
\end{assumption}

\begin{lemma}\label{petal:lem}
  There is $A \in \Lambda^+$ such that $\crk_\Lambda(K/A) = r = \crk(\Lambda)$.
\end{lemma}
\begin{proof}[Proof (cf. Proposition 10.1(7) in \cite{prdf})]
  If $r = 1$, then we can take an $A \in \Lambda$ other than $0$ and
  $K$, by the Non-degeneracy Axiom.  Then $\crk(K/A) \ge 1$, because $A
  < K$.

  Suppose $r > 1$.  Let $A$ be the base of a strict $r$-cube in
  $\Lambda$.  By Proposition~9.15 in \cite{prdf}, there exists a sequence $B_1, B_2, \ldots, B_r$ in
  $\Lambda$ such that each $B_i > A$, and the $B_i$ are
  ``independent'' over $A$, meaning that $(B_1 + \cdots + B_i) \cap
  B_{i+1} = A$ for all $i < r$.  Taking $i = 1$, we see that
  \[ B_1 \cap B_2 = A.\]
  If $A = 0$, this contradicts the Intersection Axiom of golden
  lattices.  Therefore $A > 0$, and $A \in \Lambda^+$.  Then the
  strict $r$-cube shows $\crk(K/A) = r$.
\end{proof}
\begin{definition}
  A finite set $S \subseteq K$ is said to \emph{guard} a group $A \in
  \Lambda$ if for every $B \in \Lambda$,
  \begin{equation*}
    B \supseteq S \implies B \ge A.
  \end{equation*}
\end{definition}
\begin{lemma}\label{guard:lem}
  If $A \in \Lambda$ and $\crk_\Lambda(K/A) = r$, then $A$ is guarded by some
  finite set $S$.
\end{lemma}
\begin{proof}[Proof (cf. Proposition~10.4(2) in \cite{prdf})]
  Increasing $A$, we may assume that $A$ is the base of a strict
  $r$-cube in $\Lambda$.  By Proposition~9.15 in \cite{prdf}, there
  are $B_1, B_2, \ldots, B_r \in \Lambda$ such that each $B_i > A$,
  and the $B_i$ are independent over $A$, in the sense that
  \begin{equation*}
    (B_1 + \cdots + B_{i-1}) \cap B_i = A
  \end{equation*}
  for $1 \le i < r$.  Take $g_i \in B_i \setminus A$, and let $S =
  \{g_1,\ldots,g_r\}$.  We claim $S$ guards $A$.  Suppose $C \in
  \Lambda$ and $C \supseteq S$.  Let $B'_i = (C + A) \cap B_i$.  Then
  $B'_i \le B_i$, $B'_i \le C + A$, and $A \le B'_i$.  Moreover, $g_i
  \in S \subseteq C \subseteq C+A$, and $g_i \in B_i$, so $g_i \in
  B'_i$.  As $g_i \notin A$, it follows that $A < B'_i$.  For $1 \le i
  < r$, we have
  \begin{align*}
    A &\le (B'_1 + \cdots + B'_{i-1}) \cap B'_i \le (B_1 + \cdots + B_{i-1}) \cap B_i = A \\
    A &= (B'_1 + \cdots + B'_{i-1}) \cap B'_i.
  \end{align*}
  So the $B'_i$ are an independent sequence in the interval $[A,C+A]
  \subseteq \Lambda$.  As the $B'_i$ are strictly greater than $A$, it
  follows that
  \[ \crk_\Lambda((C+A)/A) \ge r.\]
  On the other hand, by properties of cube rank (Proposition~9.28(1) in \cite{prdf}) we
  have
  \begin{align*}
    r &= \crk(\Lambda) \ge \crk_\Lambda((C+A)/(C \cap A)) =
    \crk_\Lambda((C+A)/A) + \crk_\Lambda(A/(A \cap C)) \\ & \ge r + \crk_\Lambda(A/(A \cap C)).
  \end{align*}
  Therefore $\crk_\Lambda(A/(A \cap C)) = 0$, implying $A = A \cap C$.  Thus $C
  \supseteq A$.  This shows that $A$ is guarded by $S$.
\end{proof}
\begin{lemma}\label{stuffsack:lem}
  If $S$ is a finite set and $A \in \Lambda^+$, there is $c \in
  K^\times$ such that $c S \subseteq A$.
\end{lemma}
\begin{proof}
  Let $S = \{b_1,\ldots,b_n\}$.  Each $b_i^{-1}A$ is an element of
  $\Lambda^+$, by the Scaling Axiom.  By the Intersection Axiom,
  $\bigcap_{i = 1}^n b_i^{-1} A$ is in $\Lambda^+$, hence non-zero.
  Take non-zero $c \in \bigcap_{i = 1}^n b_i^{-1} A$.  Then $cb_i \in
  A$ for all $i$.  Equivalently, $cS \subseteq A$.
\end{proof}

\begin{theorem}\label{gold:thm}
  If $\Lambda$ is a golden lattice on $K$, then $\Lambda^+ = \Lambda
  \setminus \{0\}$ is a neighborhood basis of a $W$-topology on $K$.
  If $\Lambda$ has rank $r$, then the topology is a $W_r$-topology.
\end{theorem}
\begin{proof}
  We check the relevant local sentences (copied straight out of
  \cite{PZ}).  The variables $U,V,W$ will range over $\Lambda^+$.

  First, $\Lambda^+$ is a filter base, by the Intersection Axiom:
  \[ \forall U ~ \forall V ~ \exists W : W \subseteq U \cap V.\]
  
  Second, we verify non-discreteness:
  \begin{equation}
    \forall U : \{0\} \subsetneq U. \label{eq:nondisco}
  \end{equation}
  This holds by definition of $\Lambda^+$.

  Third, we check Hausdorffness:
  \begin{equation}
    \forall x \in K^\times ~ \exists V : x \notin V. \label{eq:haus}
  \end{equation}
  By the Non-degeneracy Axiom, there is some $V_0 \in
  \Lambda$ such that $0 < V_0 < K$.  Then $V_0 \in \Lambda^+$, and
  there is some $x_0 \in K^\times$ such that $x_0 \notin V_0$.  For
  any $x \in K^\times$, we have
  \[ x = (xx_0^{-1})x_0 \notin
  (xx_0^{-1})V_0 \in \Lambda^+, \]
  by the Scaling Axiom.
  
  Next, we check continuity of addition and subtraction:
  \begin{equation}
    \forall U ~ \exists V : V - V \subseteq U. \label{subt}
  \end{equation}
  Indeed, we can take $V = U$, since the elements of $\Lambda$ are
  subgroups.

  Equations (\ref{eq:nondisco}-\ref{subt}) ensure that we have a
  non-discrete Hausdorff group topology on $(K,+)$.

  Next, we check continuity of multiplication by a constant:
  \begin{equation}
    \forall U ~ \forall x ~ \exists V : x V \subseteq U. \label{mult-easy}
  \end{equation}
  If $x = 0$, we can take any $V$.  Otherwise, we take $V = x^{-1} U$,
  using the Scaling Axiom.

  Next, we check continuity of multiplication near $(0,0)$:
  \begin{equation}
    \forall U ~ \exists V : V \cdot V \subseteq U. \label{mult-hard}
  \end{equation}
  This will take a little work.  By Lemma~\ref{petal:lem}, there is
  $V_1 \in \Lambda^+$ such that $\crk_\Lambda(K/V_1) = r$.  By
  Lemma~\ref{guard:lem}, there is a finite set $S = \{a_1,\ldots,a_n\}
  \subseteq K$ guarding $V_1$.  Let $V_2 = \bigcap_{i = 1}^n
  a_i^{-1}U$, and $V = V_1 \cap V_2$.  By the Scaling and Intersection
  Axioms, $V_2$ and $V$ are in $\Lambda^+$.  Suppose $x, y \in V$.
  Then $x \in V_2$ and $y \in V_1$.  For $i = 1, \ldots, n$, we have
  \begin{align*}
    x &\in a_i^{-1}U \\
    a_i &\in x^{-1}U.
  \end{align*}
  Thus $S \subseteq x^{-1}U$.  Now $x^{-1}U \in \Lambda$ by the Scaling Axiom, so
  \begin{equation*}
    S \subseteq x^{-1}U \implies V_1 \subseteq x^{-1}U \implies y \in
    x^{-1}U \implies xy \in U,
  \end{equation*}
  as $S$ guards $V_1$.  As $x, y$ were arbitrary elements of $V$, we
  have shown (\ref{mult-hard}).

  Equations (\ref{mult-easy}-\ref{mult-hard}) now show that
  $\Lambda^+$ defines a ring topology on $K$.

  Next we check that the ring topology is locally
  bounded\footnote{This differs from the local sentence given between
    Lemma~2.1 and Theorem~2.2 in \cite{PZ}.  But it is equivalent,
    by Lemma 2.1(d).}
  \begin{equation}
    \exists U \forall V \exists c \in K^\times : c U \subseteq V. \label{loc-bounded}
  \end{equation}
  To verify this, use Lemma~\ref{petal:lem} to find $U \in \Lambda^+$
  such that $\crk_\Lambda(K/U) = r$.  By Lemma~\ref{guard:lem}, there
  is a finite set $S \subseteq K$ that guards $U$.  Given any $V \in
  \Lambda^+$, Lemma~\ref{stuffsack:lem} gives $c \in K^\times$ such
  that $cS \subseteq V$.  Then
  \begin{equation*}
    cS \subseteq V \iff S \subseteq c^{-1}V \implies U \subseteq c^{-1}V \iff cU \subseteq V.
  \end{equation*}
  This proves (\ref{loc-bounded}).

  Lastly, we must verify the $W_r$-condition.  Take $U$ as in the
  proof of (\ref{loc-bounded}), with $U \in \Lambda^+$, and
  $\crk_\Lambda(K/U) = r$.  Then $U$ is a bounded neighborhood of 0.
  After rescaling, we may assume $1 \in U$.  By Remark~\ref{suff:rem},
  it suffices to show that $U$ is a $W_r$-set.  Let $a_1, \ldots,
  a_{r+1}$ be elements of $K$.  Because $\Lambda$ has rank $r$, there
  is some $i$ such that
  \begin{equation*}
    a_1U + \cdots + a_{r+1}U = a_1U + \cdots + a_{i-1}U + a_{i+1}U + \cdots + a_{r+1}U.
  \end{equation*}
  (See Proposition~6.3 in \cite{prdf3}.)  Then
  \begin{equation*}
    a_i = a_i 1 \in a_i U \subseteq \sum_{j = 1}^{r+1} a_j U = \sum_{j \ne i} a_j U.
  \end{equation*}
  As the $a_i$'s were arbitrary, we have shown that $U$ is a $W_r$-set.
\end{proof}

From the proof of Theorem~\ref{gold:thm}, we extract the following useful fact:
\begin{lemma}\label{pedbound:lem}
  Let $\Lambda$ be a golden lattice on $K$, of rank $r$.  If $A \in
  \Lambda$ and $\crk_\Lambda(K/A) = r$, then $A$ is bounded in the
  topology induced by $\Lambda$.
\end{lemma}

\section{Dp-finite fields}
Let $T$ be a complete, dp-finite, unstable theory of fields (possibly
with extra structure).

\begin{proposition}\label{win:prop}
  Let $\Kk$ be a highly saturated monster model of $T$.
  \begin{itemize}
  \item There is a small field $K \preceq \Kk$ such that the group
    $J_K$ of $K$-infinitesimals is co-embeddable with a definable set
    $D$.
  \item The canonical topology on $\Kk$ is a definable $W_n$-topology.
  \end{itemize}
\end{proposition}
\begin{proof}
  Let $k_0 \preceq \Kk$ be a \emph{magic subfield} (Definition~8.3 in
  \cite{prdf2}), meaning that for every type-definable $k_0$-linear
  subspace $G \subseteq \Kk$, we have $G = G^{00}$.  Let $\Lambda$ be
  the lattice of type-definable $k_0$-linear subspaces of $\Kk$.  By
  Proposition~10.1(1,2,6,7) in \cite{prdf}, this lattice is a golden lattice
  (Definition~\ref{gold:def}).  Let $\Lambda^+$ be the non-zero
  elements of $\Lambda$.  Let $n$ be $\crk(\Lambda) \le \dpr(\Kk)$.
  By Theorem~\ref{gold:thm}, $\Lambda^+$ is a neighborhood basis of 0
  for some $W_r$-topology $\tau$ on $\Kk$.  Take $V \in \Lambda^+$ a
  bounded neighborhood of 0.  By Proposition~10.1(4) in \cite{prdf}, there is a small
  field $K$ such that $J_K \in \Lambda^+$ and $J_K \subseteq V$.  Then
  $J_K$ is a bounded neighborhood of 0.  By
  Proposition~\ref{upgrader:prop}, $J_K$ is co-embeddable with a
  definable set $D$, and the $W_r$-topology $\tau$ is defined by $D$.
  It remains to show that $\tau$ is the canonical topology.
  
  After rescaling $D$, we may assume
  \begin{equation*}
    J_K \subseteq D \subseteq e J_K
  \end{equation*}
  for some $e \in K^\times$.  Now $J_K$ is a filtered intersection of
  the $K$-definable canonical basic neighborhoods.  Shrinking $D$, we may assume that
  $D$ is a $K$-definable canonical basic neighborhood.
  \begin{claim}\label{cl7}
    If $U$ is a $K$-definable canonical basic neighborhood, then there
    is $c \in K^\times$ such that $c D \subseteq U$.
  \end{claim}
  \begin{claimproof}
    The sets $U$ and $D$ are definable over $K$, and $K \preceq \Kk$,
    so it suffices to find $c \in \Kk^\times$.  Take $c = e^{-1}$:
    \begin{equation*}
      e^{-1} D \subseteq J_K \subseteq U.
      \qedhere
    \end{equation*}
  \end{claimproof}
  Claim~\ref{cl7} is a conjunction of first-order sentences, so it
  transfers from $K$ to $\Kk$.  Therefore, if $U$ is any
  $\Kk$-definable canonical basic neighborhood, then there is $c \in
  \Kk^\times$ such that $c D \subseteq U$.  It follows that $D$
  defines the canonical topology on $\Kk$, which must agree with the
  definable $W_n$-topology $\tau$.
\end{proof}

\begin{theorem}
  If $K$ is an unstable field with $\dpr(K) = n$, then the canonical
  topology on $K$ is a definable $W_n$-topology.
\end{theorem}
\begin{proof}
  The proof of (\cite{prdf4}, Theorem~6.27) applies here.
\end{proof}

Theorem~\ref{defcourse} then yields
\begin{corollary}\label{wow:cor}
  Every unstable dp-finite field admits a definable V-topology.
\end{corollary}
By Proposition~3.5 in \cite{hhj-v-top}, definable V-topologies on
sufficiently saturated fields are induced by externally definable
valuation rings.  Given that the Shelah expansion of a dp-finite
structure is dp-finite, we conclude the following:
\begin{corollary}\label{hS:cor}
  The henselianity conjecture for dp-finite fields implies the Shelah
  conjecture for dp-finite fields.
\end{corollary}
This gives a smoother proof of the Shelah conjecture for positive
characteristic dp-finite fields, where the henselianity conjecture is
known (Theorem~2.8 in \cite{prdf}).

We can say the following more precise version of
Corollary~\ref{wow:cor}:
\begin{theorem}\label{insight:thm}
  Let $K$ be an unstable dp-finite field.  The definable V-topologies
  on $K$ are exactly the V-topological coarsenings of the canonical
  topology on $K$.
\end{theorem}
\begin{proof}
  Let $\tau_0$ be the canonical topology.  If $\tau$ is a
  V-topological coarsening of $\tau_0$, then $\tau$ is definable by
  Theorem~\ref{defcourse}.  Conversely, suppose that $\tau$ is a
  definable V-topology.  Let $B$ be a definable bounded neighborhood.
  After rescaling $B$, we may assume that for any $x, y \in K$,
  \begin{equation*}
    x \in B y \text{ or } y \in B x
  \end{equation*}
  as this is the definition of a $W_1$-topology.  Taking $y = 1$, we
  see that for any $x$,
  \begin{equation*}
    x \in B \text{ or } 1 \in B x,
  \end{equation*}
  or equivalently, $B$ contains one of $x$ or $1/x$.  Then $B$ must
  have full dp-rank, as two copies of it cover $K$.  Then $B - B$ is a
  neighborhood in $\tau_0$, by Corollary~5.10 in \cite{prdf2}.  Now
  the set $B - B$ is a bounded neighborhood in $\tau$, so the family
  of sets
  \begin{equation*}
    \{a \cdot (B - B) : a \in K^\times\}
  \end{equation*}
  is a neighborhood basis of 0 for $\tau$.  All these sets are
  neighborhoods in $\tau_0$, and so $\tau$ must be coarser than
  $\tau_0$.  So $\tau$ is one of the V-topological coarsenings of
  $\tau_0$.
\end{proof}

\subsection{Three conjectures}\label{conjs:sec}

If $\tau_1, \ldots, \tau_n$ are field topologies on $K$, then they
generate a minimal common refinement $\tau$.  This topology $\tau$ is
also a field topology---except that it may be discrete.  A basis of
neighborhoods is given by
\begin{equation*}
  \{ U_1 \cap \cdots \cap U_n : U_1 \in \tau_1, ~ U_2 \in \tau_2,
  \ldots, U_n \in \tau_n\}.
\end{equation*}

\begin{conjecture}\label{dream}
  Let $(K,\tau)$ be a W-topological field of characteristic 0.  Then
  $\tau$ is generated by jointly independent topologies $\tau_1,
  \ldots, \tau_n$, and each $\tau_i$ has a \emph{unique} V-topological
  coarsening.
\end{conjecture}
The only real evidence for Conjecture~\ref{dream} is the
classification of $W_2$-topologies in \S\ref{w2-fun} below.
\begin{theorem}
  Conjecture~\ref{dream} implies the Shelah conjecture for dp-finite
  fields.
\end{theorem}
\begin{proof}
  By Corollary~\ref{hS:cor}, it suffices to prove the henselianity
  conjecture for dp-finite fields.  Suppose the henselianity
  conjecture fails.  By the usual techniques\footnote{See the proofs
    of Lemmas~9.7, 9.8 in \cite{prdf2}, or Propositions~6.3, 6.4 in
    \cite{prdf4}.}, we get a dp-finite multivalued field
  $(K,\Oo_1,\Oo_2)$, where $\Oo_1$ and $\Oo_2$ are independent
  non-trivial valuation rings.  By Lemma~2.6 in \cite{prdf}, $K$ has
  characteristic 0.  Let $\tau$ be the canonical topology on the
  structure $(K,\Oo_1,\Oo_2)$.  Note that $\Oo_1$ and $\Oo_2$ define
  two distinct V-topological coarsenings of $\tau$, by
  Theorem~\ref{insight:thm}.

  Applying Conjecture~\ref{dream} to $\tau$, we decompose $\tau$ into
  independent $\tau_1, \ldots, \tau_m$.  We claim $m > 1$.  Otherwise,
  $\tau = \tau_1$, and then $\tau$ has a unique V-topological
  coarsening, a contradiction.

  So $m \ge 2$.  Because each $\tau_i$ is Hausdorff, there are $U_i
  \in \tau_i$ such that $(1 + U_i) \cap (-1 + U_i)$.  Let $U = U_1
  \cap \cdots \cap U_m$.  Then $U \in \tau$.
  \begin{claim}\label{oppo:claim}
    For every $V \in \tau$, we have $1 + V \not \subseteq (1 + U)^2$.
  \end{claim}
  \begin{claimproof}
    Shrinking $V$, we may assume $V = V_1 \cap \cdots \cap V_n$, where
    each $V_i \in \tau_i$.  By continuity of multiplication, there are
    $W_i \in \tau_i$ such that $(1 + W_i)^2 \subseteq 1 + V_i$.
    Shrinking the $W_i$, we may assume $W_i = - W_i$ and $W_i
    \subseteq U_i$.  By joint independence of the $\tau_i$, we can
    find
    \[ x \in (-1 + W_1) \cap (1 + W_2) \cap \cdots \cap (1 + W_m).\]
    Then $\pm x \in 1 + W_i$ for all $i$, so $x^2 \in 1 + V_i$ for all
    $i$.  Thus $x^2 \in 1 + V$.  On the other hand, $x^2 \notin (1 +
    U)^2$.  Otherwise, one of $x$ or $-x$ is in $1 + U$.
    \begin{itemize}
    \item If $x \in U$, then $x \in 1 + U_1$.  But $x \in -1 + W_1
      \subseteq -1 + U_1$.  So the two sets $1 + U_1$ and $-1 + U_1$
      fail to be disjoint.
    \item If $-x \in U$, then $-x \in 1 + U_2$.  But $-x \in -1 - W_2
      \subseteq -1 + U_2$.  Then the two sets $1 + U_2$ and $-1 + U_2$
      fail to be disjoint.
    \end{itemize}
    Either way, this contradicts the choice of the $U_i$.
  \end{claimproof}
  Let $\Kk$ be a saturated elementary extension of $K$, and let
  $J_{\Kk}$ be the group of $K$-infinitesimals.  By Proposition~5.17(4) in
  \cite{prdf2}, $1 + J_K \subseteq (1 + J_K)^2$, where $(1 + J_K)^2$
  denotes the image of $1 + J_K$ under the squaring map.  The two sets
  can be written as filtered intersections
  \begin{align*}
    1 + J_K &= \bigcap \{ 1 + V : V \text{ a $K$-definable canonical
      basic neighborhood}\} \\
    (1 + J_K)^2 &= \bigcap \{ (1 + U)^2 : U \text{ a $K$-definable
      canonical basic neighborhood}\}.
  \end{align*}
  Therefore the following local sentence holds in $K$ with its
  canonical topology $\tau$:
  \begin{equation*}
    \forall U \in \tau \exists V \in \tau : 1 + V \subseteq (1 + U)^2.
  \end{equation*}
  This contradicts Claim~\ref{oppo:claim}.
\end{proof}

\begin{conjecture}\label{perfect:conj}
  If $K$ is a perfect field of positive characteristic, and $\tau$ is
  a $W_n$-topology on $K$, then $\tau$ is generated by $n$ independent
  V-topologies.
\end{conjecture}
This would imply the positive-characteristic case of the
``valuation-type conjecture'' (Conjecture~10.1 in \cite{prdf2}), which
says that the canonical topology on an unstable dp-finite field is a
V-topology.  This is false in characteristic 0 (\S 10 in
\cite{prdf4}), but may still hold in positive characteristic.

The evidence for Conjecture~\ref{perfect:conj} is that it holds for
$W_2$-topologies in odd characteristic, by combining Proposition~5.32
of \cite{prdf4} with the methods of \S\ref{w2-fun} below.

\begin{conjecture}\label{burden}
  If $R$ is a $W_n$-ring on an algebraically closed field, then $R$ is
  NTP$_2$, and the burden of $R$ is at most $n$.
\end{conjecture}
For example, this holds for
\begin{itemize}
\item Multivaluation rings, by \cite{acf-multival}.
\item The diffeovaluation $W_2$-rings constructed in \cite{prdf4},
  specifically the rings $Q$ and $R$ of \S 8.4.  These are $W_2$-rings
  by (\cite{prdf4}, Lemma~8.23), and have burden $\le 2$ by
  (\cite{prdf4}, Theorem~10.24).
\end{itemize}

\section{Miscellaneous results}
We prove a few miscellaneous results.  In \S\ref{coarse:sec}, we show
that coarsenings of W-topologies are W-topologies.  Moreover, weight
must decrease in a strict coarsening.  This implies bounds on the
length of chains of coarsenings.

In \S\ref{vn:sec}, we consider topologies generated by $n$ independent
V-topologies.  We show that
\begin{itemize}
\item The class of such field topologies is a local class.  (This is
  probably well-known to experts, and useful in
  \S\ref{sec:inflators}.)
\item Every such topology has weight $n$.
\end{itemize}

Lastly, in \S\ref{approx:sec} we prove a very weak analogue of the
approximation theorem for V-topologies.  This may be useful in
attacking Conjecture~\ref{dream}.

\subsection{Coarsenings of W-topologies}\label{coarse:sec}
\begin{lemma}\label{coarse1:lem}
  Let $\tau, \tau'$ be two ring topologies on $K$, with $\tau'$
  coarser than $\tau$ (i.e., $\tau' \subseteq \tau$).  If $\tau$ is a
  $W_n$-topology, then $\tau'$ is a $W_n$-topology.\footnote{This
    looks easy, given that the local sentence appearing in
    Remark~\ref{local-class:rem} is preserved in coarsenings.  But the
    difficulty is showing that $\tau'$ is \emph{locally bounded} in
    the first place.  A coarsening of a locally bounded ring topology
    need not be locally bounded.  For example, the diagonal embedding
    of $\Qq$ into $\prod_p \Qq_p$ induces a field topology on $\Qq$
    that is not locally bounded.  This topology is a coarsening of the
    locally bounded ring topology $\tau_\Zz$ induced by $\Zz$.}
\end{lemma}
\begin{proof}
  Suppose not.  There are local sentences expressing that $\tau'$ is
  coarser than $\tau$, that $\tau$ is a $W_n$-topology, and that
  $\tau'$ is \emph{not} a $W_n$-topology.  By Theorem~1.1 (and
  Remark~1.5) of \cite{PZ}, we may assume that $(K,\tau,\tau')$ is
  $\omega$-complete.  By Lemma~\ref{extend:lem}, $\tau$ is induced by a $W_n$-ring $R$
  on $K$.

  We claim that $R$ is $\tau'$-bounded.  Indeed, if $U \in \tau'$,
  then $U \in \tau$, so there is $c \in K^\times$ such that $cR
  \subseteq U$.  By Lemma~2.1(d) in \cite{PZ}, the $\tau'$-boundedness
  of $R$ means that
  \begin{equation}
    \forall U \in \tau' ~ \exists V \in \tau' : V \cdot R \subseteq U. \label{dag}
  \end{equation}
  \begin{claim}\label{module-down}
    For every $U \in \tau'$, there exists smaller $M \in \tau'$ such
    that $M$ is an $R$-submodule of $K$.
  \end{claim}
  \begin{claimproof}
    Define a descending sequence of $\tau'$-neighborhoods $U = U_0
    \supseteq U_1 \supseteq U_2 \supseteq \cdots$, choosing $U_{n+1}$
    small enough that
    \begin{equation*}
      U_{n+1} \cup (U_{n+1} - U_{n+1}) \cup (U_{n+1} \cdot R) \subseteq U_n.
    \end{equation*}
    This is possible using (\ref{dag}) and the fact that $\tau'$ is a
    group topology on $(K,+)$.  Let $M = \bigcap_n U_n$.  Then $M
    \subseteq U_0 = U$, and $M \in \tau'$ by $\omega$-completeness.
    Lastly, $M$ is an $R$-submodule of $K$ by construction.
  \end{claimproof}
  Let $\Lambda^+$ be the set of all $R$-submodules of $K$ which are
  $\tau'$-neighborhoods of 0.  Let $\Lambda = \Lambda^+ \cup \{0\}$.
  Then $\Lambda$ is a golden lattice (Definition~\ref{gold:def}):
  \begin{itemize}
  \item $\Lambda^+$ is clearly closed under intersections and joins,
    proving the Lattice and Intersection Axioms.
  \item The Scaling Axiom holds because $\tau'$ is a ring topology.
  \item The Rank Axiom holds because $\Sub_R(M)$ has finite rank, and
    $\Lambda$ is a sublattice.
  \item The Non-degeneracy Axiom holds by applying
    Claim~\ref{module-down} to any neighborhood $U \in \tau'$ strictly
    smaller than $K$.
  \end{itemize}
  By Theorem~\ref{gold:thm}, $\Lambda^+$-defines a $W_n$-topology
  $\tau''$ on $K$.  Then $\tau'' \subseteq \tau'$, by definition of
  $\Lambda^+$.  On the other hand, $\tau' \subseteq \tau''$ by
  Claim~\ref{module-down}.  Thus $\tau'$ is a $W_n$-topology, a
  contradiction.
\end{proof}

\begin{lemma}\label{bump:lem}
  Let $R \subseteq R'$ be two rings on $K = \Frac(R)$.  If $R$ is a
  $W_n$-ring, then
  \begin{itemize}
  \item $R'$ is a $W_{n-1}$-ring, \emph{or}
  \item $R$ and $R'$ are co-embeddable.
  \end{itemize}
\end{lemma}
\begin{proof}
  Decreasing $n$, we may assume $n = \wt(R) = \crk_R(K)$.  By
  Lemma~\ref{ring-change:lem}, $\crk_{R'}(K) \le n$, so we may assume
  $\wt(R') = \crk_{R'}(K) = n$.  Then we must show that $R$ and $R'$
  are co-embeddable.  Let $\Lambda$ and $\Lambda'$ be the lattices of
  $R$-submodules and $R'$-submodules of $K$.  Then $\Lambda$ and
  $\Lambda'$ are golden lattices (Definition~\ref{gold:def}).  Also,
  $\Lambda'$ is a sublattice of $\Lambda$.  By Lemma~\ref{petal:lem},
  there is $A \in \Lambda'$ such that $A$ is the base of a strict
  $n$-cube in $\Lambda'$.  Then $A$ is the base of a strict $n$-cube
  in $\Lambda$ as well.  Then \[ \crk_R(K/A) = \crk_{R'}(K/A) = n =
  \crk_R(K)= \crk_{R'}(K).\] By Lemma~\ref{pedbound:lem}, $A$ is a
  bounded neighborhood in both $\tau_R$ and $\tau_{R'}$.  Then $R$ is
  co-embeddable with $A$, and $A$ is co-embeddable with $R'$.
\end{proof}

\begin{proposition}\label{bump:prop}
  If $\tau$ is a $W_n$-topology on a field $K$, and $\tau'$ is a
  strict coarsening, then $\tau'$ is a $W_{n-1}$-topology on $K$.
\end{proposition}
\begin{proof}
  As in Lemma~\ref{coarse1:lem}, we may assume $(K,\tau,\tau')$ is
  $\omega$-complete.  Then $\tau$ is induced by a $W_n$-ring $R$.  The
  ring $R$ is $\tau'$-bounded.  By Lemma~\ref{extend:lem}, $\tau'$ is
  induced by some superring $R' \supseteq R$.  Then
  Lemma~\ref{bump:lem} implies one of the following:
  \begin{itemize}
  \item $R'$ is a $W_{n-1}$-ring, implying that $\tau'$ is a
    $W_{n-1}$-ring.
  \item $R$ and $R'$ are co-embeddable, implying that $\tau =
    \tau'$. \qedhere
  \end{itemize}
\end{proof}

\subsection{$V^n$-topologies} \label{vn:sec}
\begin{definition}
  A \emph{$V^n$-topology} on $K$ is a locally bounded ring topology $\tau$
  such that the following local sentence holds: there are distinct
  $q_1, \ldots, q_n \in K$ such that for any $U \in \tau$, there is $c
  \in K^\times$ such that for all $x \in K$,
  \begin{equation*}
    (\{x\} \cup \{1/(x - q_i) : 1 \le i \le n\}) \cap cU \ne \emptyset.
  \end{equation*}
\end{definition}
\begin{remark}\label{v:rem}
  An equivalent condition is that there is a bounded neighborhood $U
  \ni 0$ and elements $q_1, \ldots, q_n$ such that for every $x \in
  K$,
  \begin{equation*}
    (\{x\} \cup \{1/(x - q_i) : 1 \le i \le n\}) \cap U \ne \emptyset.
  \end{equation*}
\end{remark}
Note that $V^n$-topologies form a local class.
\emph{Non}-$V^n$-topologies form a local class as well.
\begin{lemma}\label{v1:lem}
  If $R$ is an intersection of $n$ valuation rings on $K$, then $R$ induces a
  $V^n$-topology on $K$.
\end{lemma}
\begin{proof}
  First of all, $R$ defines a locally bounded field topology because
  $\Frac(R) = K$ by (\cite{prdf2}, Proposition~6.2(3)).  Suppose $\tau_R$ fails to
  be a $V^n$-topology.  Passing to an elementary extension, we may
  assume that $\tau_R$ is $\omega$-complete.  Let $K_0$ be a subfield
  of $K$ of size $\aleph_0$.  Then $K_0$ is bounded, by
  $\omega$-completeness.  Let $R'$ be the ring generated by $R$ and
  $K_0$.  This ring continues to define $\tau_R$, by
  Proposition~\ref{gap:prop}.  Also $R'$ is a multivaluation ring, by
  Proposition~6.10 in \cite{prdf2}.  If $q_1, \ldots, q_n$ are arbitrary distinct
  elements of $K_0$, and if $x \in K$, then
  \begin{equation*}
    (\{x\} \cup \{1/(x - q_i) : 1 \le i \le n\}) \cap R' \ne \emptyset,
  \end{equation*}
  by Lemma~5.24 in \cite{prdf3}.
\end{proof}

\begin{lemma}\label{v2:lem}
  If $(K,\tau)$ is an $\omega$-complete $V^n$-topology, then
  $(K,\tau)$ is induced by a ring $R$ that is an intersection of $n$
  valuation rings on $K$.
\end{lemma}
\begin{proof}
  Let $U$ be a bounded neighborhood, and $q_1, \ldots, q_n$ be
  as in Remark~\ref{v:rem}, so that for any $x \in K$,
  \begin{equation*}
    (\{x\} \cup \{1/(x - q_i) : 1 \le i \le n\}) \cap U \ne \emptyset.
  \end{equation*}
  Let $K_0$ be a countable subfield containing the $q_i$.  Then $K_0$
  is bounded, by $\omega$-completeness.  By
  Proposition~\ref{gap:prop}, $\tau$ is induced by some subring $R$
  containing $U$ and $K_0$.  Then $R$ is a $K_0$-algebra, and for
  every $x \in K$, at least one of the numbers
  $x,1/(x-q_1),\ldots,1/(x-q_n)$ lies in $R$.  Then $R$ is an
  intersection of $n$ valuation rings, by Lemma~5.24 in \cite{prdf3}.
\end{proof}

\begin{corollary}\label{vbasic:cor}
  Let $K$ be a field with a ring topology $\tau$.
  \begin{enumerate}
  \item $\tau$ is a $V^n$-topology if and only if $(K,\tau)$ is
    locally equivalent to a field with a topology induced by an
    intersection of $n$ valuation rings.
  \item If $\tau$ is a $V^n$-topology, then $\tau$ is a
    $W_n$-topology, and hence a field topology.
  \item If $\tau$ is a $V^n$-topology, then $\tau$ is a $V^m$-topology
    for $m > n$.
  \item $\tau$ is a $V^1$-topology if and only if $\tau$ is a
    V-topology.
  \end{enumerate}
\end{corollary}

\begin{proposition}
  If $\tau_1, \ldots, \tau_n$ are distinct V-topologies on a field
  $K$, and $\tau$ is the topology generated by $\tau_1, \ldots,
  \tau_n$ (as in Corollary~4.3 of \cite{PZ}), then $\tau$ is a
  $V^n$-topology.  On the other hand, $\tau$ is not a
  $W_{n-1}$-topology, and therefore not a $V^{n-1}$-topology.
\end{proposition}
\begin{proof}
  As in the proof of Theorem~4.4 of \cite{PZ}, we may assume that
  $(K,\tau_1,\ldots,\tau_n,\tau)$ is $\omega$-complete.  Then each
  $\tau_i$ is induced by a valuation ring $\Oo_i$.  One easily sees
  that the topology $\tau$ is induced by $R = \Oo_1 \cap \cdots \cap
  \Oo_n$.  By Corollary~\ref{vbasic:cor}, $\tau$ is a $V^n$-topology.

  Now suppose that $\tau$ is a $W_{n-1}$-topology.  Then some
  $W_{n-1}$-set $U$ is a bounded neighborhood of 0 with respect to
  $\tau$.  The set $U$ is bounded with respect to the coarser
  topologies $\tau_i$.  By Proposition~\ref{gap:prop}, we may coarsen
  the $\Oo_i$ to ensure that $U \subseteq \Oo_i$ for each $i$.  Then
  $R = \Oo_1 \cap \cdots \cap \Oo_n$ contains $U$, hence is a
  $W_{n-1}$-ring.  This contradicts Proposition~\ref{multi:prop}.
  (The $\Oo_i$ are pairwise incomparable, because they are
  \emph{independent}.)
\end{proof}

\begin{proposition}
  Let $\tau$ be a $V^n$-topology on a field $K$.  Then $\tau$ is
  generated by $n$ or fewer V-topologies on $K$.
\end{proposition}
\begin{proof}
  Let $q_1, \ldots, q_n$ and $U$ be as in Remark~\ref{v:rem}, so $U$
  is a bounded neighborhood of 0 and
  \begin{equation*}
    \forall x \in K : U \cap \{x,1/(x-q_1),\ldots,1/(x-q_n)\} \ne \emptyset.
  \end{equation*}
  Let $(K^*,U^*)$ be a saturated elementary extension of $(K,U)$, and
  let $\tau^*$ be the topology induced by $U^*$.  Let $K_0$ be the
  countable subfield of $K$ generated by the $q_i$.  Let $R \subseteq
  K^*$ be the ring generated by $K_0$ and $U^*$.  Note that $R$ is
  $\vee$-definable over $K_0$.  As in the proof of Lemma~\ref{v2:lem},
  the ring $R$ is a multivaluation ring inducing $\tau^*$.  Therefore
  $R$ and $U^*$ are co-embeddable.  Let $\pp_1, \ldots, \pp_m$ be the
  maximal ideals of $R$; $m$ is the number of valuation rings needed
  to define $R$, so $m \le n$.  As in the proof of
  Theorem~\ref{defcourse}.\ref{p1}, each localization $R_{\pp_i}$ is a
  valuation ring, and the induced V-topology is definable in the
  structure $(K^*,U^*)$.  The resulting V-topologies generate
  $\tau^*$, because $R = \bigcap_i R_{\pp_i}$.  Then the statement
  \begin{quote}
    ``$\tau$ is the topology generated by $m$ distinct definable
    V-topologies in $(K^*,U^*)$''
  \end{quote}
  can be expressed by a disjunction of first-order sentences.
  Therefore it transfers to the elementary substructure $(K,U)$.
\end{proof}

We summarize the situation below:
\begin{theorem}\label{vsum:thm}
  For every $n$, there is a local sentence $\sigma_n$ holding in
  $(K,\tau)$ if and only if $\tau$ is generated by $n$ independent
  V-topologies.  If $(K,\tau) \models \sigma_n$, then $(K,\tau)$ is a
  $W_n$-topology, but not a $W_{n-1}$-topology.
\end{theorem}

\begin{remark}\label{kowalsky}
  A coarsening of a $V^n$-topology is again a $V^n$-topology.  This is
  Lemma~4.4 in \cite{PZ}.
\end{remark}

\subsection{Independence/approximation}\label{approx:sec}

\begin{definition}
  Let $\tau, \tau'$ be two ring topologies on $K$.  Then $\tau$ and
  $\tau'$ are \emph{independent} if every $\tau$-open set $U$
  intersects every $\tau'$-open set $V$.
\end{definition}
This can be expressed via a local sentence:
\begin{equation*}
  \forall x ~\forall y ~\forall U \in \tau ~\forall V \in \tau' ~\exists z
  : (z - x \in U \text{ and } z - y \in V).
\end{equation*}
Note that we can equivalently just say
\begin{equation*}
  \forall U \in \tau ~ \forall V \in \tau' : U + V = K.
\end{equation*}

\begin{lemma}\label{ind:lem}
  Let $R$ be a $W_n$-ring on $K$.  For $i = 1, 2$, let $R_i$ be a
  subring of $K$ containing $R$.  Then one of the following
  holds:
  \begin{itemize}
  \item There is a V-topology coarser than both $\tau_{R_1}$ and
    $\tau_{R_2}$.
  \item $\tau_{R_1}$ and $\tau_{R_2}$ are independent.
  \end{itemize}
\end{lemma}
\begin{proof}
  Let $\Lambda^+$ be the set of $R$-submodules of $M \le K$ satisfying
  the following equivalent conditions:
  \begin{itemize}
  \item $M$ is a neighborhood of 0 in both $\tau_{R_1}$ and
    $\tau_{R_2}$.
  \item There are non-zero ideals $I_1 \le R_1$ and $I_2 \le R_2$ such
    that $I_1, I_2 \subseteq M$.
  \item There are non-zero $c_1, c_2 \in K$ such that $c_1 R_1
    \subseteq M$ and $c_2 R_2 \subseteq M$.
  \end{itemize}
  Then $\Lambda^+$ is an unbounded sublattice of $\Sub_R(K)$, closed
  under scaling by $K^\times$.  Let $\Lambda = \{0\} \cup \Lambda^+$.
  Then $\Lambda$ is a bounded sublattice of $\Sub_R(K)$.  It has rank
  at most $n$.  Thus $\Lambda$ satisfies all the axioms of golden
  lattices (Definition~\ref{gold:def}), except possibly
  non-degeneracy.

  If $\tau_{R_1}$ and $\tau_{R_2}$ are independent, we are done.
  Otherwise, there are ideals $I_1 \le R_1$ and $I_2 \le R_2$ such
  that $I_1 + I_2 < K$.  Then $I_1 + I_2 \in \Lambda$, showing that
  $\Lambda$ is a golden lattice.  By Theorem~\ref{gold:thm}, the sets
  $\Lambda^+$ define a $W_n$-topology $\tau'$ on $K$.  By definition
  of $\Lambda$, every $\tau'$-neighborhood of 0 is a neighborhood of 0
  in the topologies $\tau_{R_1}$ and $\tau_{R_2}$.  Thus $\tau'$ is a
  common coarsening of $\tau_{R_1}$ and $\tau_{R_2}$.  By
  theorem~\ref{defcourse}, there is a V-topology coarser than $\tau'$,
  hence coarser than $\tau_{R_1}$ and $\tau_{R_2}$.
\end{proof}

\begin{theorem}
  Let $\tau_0, \tau_1, \tau_2$ be three W-topologies on $K$, with
  $\tau_0$ finer than $\tau_1$ and $\tau_2$.  Then at least one of the
  following holds:
  \begin{itemize}
  \item $\tau_1$ and $\tau_2$ are independent.
  \item $\tau_1$ and $\tau_2$ share a common V-topological coarsening.
  \end{itemize}
\end{theorem}
\begin{proof}
  Suppose $\tau_1$ and $\tau_2$ are dependent.  As usual, we can find
  sets $B_0, B_1, B_2$ such that
  \begin{itemize}
  \item $B_i$ is a bounded neighborhood of 0 in $\tau_i$
  \item $B_0 \subseteq B_1$ and $B_0 \subseteq B_2$
  \item $1 \in B_0$, and $B_0$ is a $W_n$-set for some $n$.
  \end{itemize}
  Let $(K^*,B_0^*,B_1^*,B_2^*)$ be a saturated elementary extension of
  $(K,B_0,B_1,B_2)$.  As usual, $B_i^*$ defines a topology $\tau_i^*$
  on $K^*$, and $(K^*,\tau_0^*,\tau_1^*,\tau_2^*)$ is locally
  equivalent to $(K,\tau_0,\tau_1,\tau_2)$.  In particular, $\tau_1^*$
  and $\tau_2^*$ are still dependent.

  Let $R_i$ be the $\vee$-definable ring generated by $B_i$.  As
  usual, $R_i$ generates $\tau_i$.  Then $R_0$ contains the $W_n$-set
  $B_0$, so $R_0$ is a $W_n$-ring.  Because $\tau_1^*$ and $\tau_2^*$
  are not independent, Lemma~\ref{ind:lem} yields a V-topology coarser
  than both $R_1$ and $R_2$.  By Theorem~\ref{defcourse}, this
  V-topology is definable in the structure $(K^*,B_0^*,B_1^*,B_2^*)$.
  The statement ``there is a definable V-topology coarser than the
  topologies induced by $B_1^*$ and $B_2^*$'' is expressed by a
  disjunction of first-order sentences, so it holds in the elementary
  substructure $(K,B_0,B_1,B_2)$.
\end{proof}

\section{$W_n$-rings and inflators} \label{sec:inflators}
\begin{lemma}
  Let $R$ be a ring and $M$ be a module.  If $\crk_R(M) \ge n$, then
  $M$ has a subquotient that is semisimple of length $n$.
\end{lemma}
\begin{proof}
  $M$ has a subquotient isomorphic to $\bigoplus_{i = 1}^n N_i$ for
  some non-trivial $R$-modules $N_i$.  Each $N_i$ has a simple
  subquotient $N'_i$.  Then $\bigoplus_{i = 1}^n N'_i$ is a
  subquotient of $M$.
\end{proof}

\begin{proposition}\label{to-inflators}
  Let $R$ be a $W_n$-ring on a field $K$.  Then there is $m \le n$ and ideals
  $A \subseteq B \subseteq R$ such that
  \begin{enumerate}
  \item $B/A$ is a semisimple $R$-module of length $m$.
  \item The induced map
    \begin{align*}
      \varsigma : \Dir_K(K) & \to \Dir_R(B/A) \\
      \Sub_K(K^i) & \to \Sub_R(B^i/A^i) \\
      V & \mapsto (V \cap B^i + A^i)/A^i
    \end{align*}
    is a malleable $m$-inflator.
  \item If $\varsigma'$ is any mutation of $\varsigma$, such as
    $\varsigma$ itself, and if $R'$ is the fundamental ring of $R$,
    then there is $c \in K^\times$ such that
    \begin{equation*}
      R \subseteq R' \subseteq c R.
    \end{equation*}
    Therefore $R'$ is a $W_n$-ring co-embeddable with $R$.
  \end{enumerate}
\end{proposition}
\begin{proof}
  Let $m = \wt(R) = \crk_R(R)$.  Take submodules (i.e., ideals) $A \subseteq B
  \subseteq R$ such that $B/A$ is semisimple of length $m$.  The
  first point holds.  Note $m = \crk_R(K)$, by Lemma~\ref{kr:lem}.
  Then $A$ is a ``pedestal'' in the lattice of $R$-submodules of $K$,
  and the second point follows by Theorems~8.5, 8.9, 8.12 in \cite{prdf3}.  If
  $\varsigma'$ is a mutation of $\varsigma$, then $\varsigma'$ is
  induced by another pedestal $A'$, of the form
  \begin{equation*}
    A' = b_1 A \cap \cdots \cap b_k A,
  \end{equation*}
  for some non-zero $b_i$, by Proposition~10.15 in \cite{prdf3}.  The
  the fundamental ring $R'$ is exactly the ``stabilizer''
  \begin{equation*}
    R' = \{x \in K : xA' \subseteq A'\},
  \end{equation*}
  by Proposition~8.10 in \cite{prdf3}.
  Certainly $R' \supseteq R$, as $A'$ is an $R$-module.  Also, $A'$ is
  a bounded neighborhood of 0, because $A$ is.  Therefore, there is
  non-zero $c_1 \in A'$, and there is non-zero $c_2 \in K^\times$ such
  that $c_2 A' \subseteq R$.  Then
  \begin{equation*}
    c_2 c_1 R' \subseteq c_2 A' R' \subseteq c_2 A' \subseteq R,
  \end{equation*}
  showing that $R'$ is embeddable into $R$.
\end{proof}

\subsection{$W_2$-rings}
Applying the results of \cite{prdf4}, we obtain the following fact
about $W_2$-rings on fields of characteristic 0:
\begin{theorem}\label{w2ring:thm}
  Let $K$ be a field of characteristic 0, and let $R$ be a $W_2$-ring
  on $K$.  Then one of two things happens:
  \begin{enumerate}
  \item $R$ is co-embeddable with a ring of the form
    \[ Q = \{x \in K : \val(x) \ge 0 \text{ and } \val(\partial x) \ge 0\} \]
    induced by some dense ``diffeovaluation data'' as in \S 8 of
    \cite{prdf4}.
  \item \label{lazy} There is a multi-valuation ring $S$ and $c \in K^\times$ such
    that $cS \subseteq R$.
  \end{enumerate}
\end{theorem}
\begin{proof}
  Let $\varsigma$ be the malleable $m$-inflator as in
  Proposition~\ref{to-inflators}.  Then $m = 1$ or $m = 2$.  We break
  into two cases:
  \begin{itemize}
  \item Suppose no mutation $\varsigma'$ of $\varsigma$ is weakly
    multi-valuation type (Definition~5.27 in \cite{prdf3}).  By
    Proposition~5.19 in \cite{prdf3}, $m > 1$, so $m = 2$.  By
    Corollary~8.27 in \cite{prdf4}, some mutation $\varsigma'$ of
    $\varsigma$ is a ``diffeovaluation inflator'' (Definition~8.25 in
    \cite{prdf4}).  Let $R'$ be the fundamental ring of $\varsigma'$.
    By Proposition~\ref{to-inflators}, $R'$ is co-embeddable with $R$.
    By the proof of Corollary~8.27 in \cite{prdf4}, $R'$ is the
    desired set $\{x \in K : \val(x) \ge 0 \text{ and } \val(\partial
    x) \ge 0\}$ obtained from the diffeovaluation data.
  \item Suppose some mutation $\varsigma'$ of $\varsigma$ is weakly
    multi-valuation type.  By definition, this means that the
    fundamental ring $R'$ of $\varsigma'$ contains $eS$ for some $e
    \in K^\times$ and some multivaluation ring $S$.  Then $S$ is
    embeddable into $R'$, and by Proposition~\ref{to-inflators}, $R'$
    is embeddable into $R$. \qedhere
  \end{itemize}
\end{proof}

\subsection{$W_2$-topologies in characteristic 0} \label{w2-fun}

\begin{lemma}\label{dv:lem}
  Let $K$ be a field of characteristic 0.  Let $\tau$ be a DV-topology
  in the sense of \cite{prdf4}, Definition~8.18.  Then $\tau$ is a
  $W_2$-topology, but not a $V^n$-topology for any $n$.
\end{lemma}
\begin{proof}
  By definition of ``DV-topology,'' $\tau$ is locally equivalent to a
  diffeovaluation topology in the sense of Definition~8.16,
  \cite{prdf4}.  The $W_2$-topologies and $V^n$-topologies are local
  classes, so we may assume $\tau$ is a diffeovaluation topology.  Let
  $Q$ and $R$ be
  \begin{align*}
    Q &= \{x \in K : \val(x) \ge 0 \text{ and }
    \val(\partial x) > 0\} \\
    R &= \{x \in K : \val(x) \ge 0 \text{ and }
    \val(\partial x) \ge 0\}
  \end{align*}
  as in \S8.4 of \cite{prdf4}.  Then $Q$ and $R$ are rings inducing
  $\tau$.  By Lemma~8.23 in \cite{prdf4}, $Q$ is a $W_2$-ring, and so
  $\tau$ is a $W_2$-topology.

  Suppose that $\tau$ is a $V^n$-topology for some $n$.  Then $\tau$
  is induced by independent V-topologies $\tau_1, \ldots, \tau_m$.
  After passing to an elementary extension of the original
  diffeovalued field, we may assume that each $\tau_i$ is induced by a
  valuation ring $\Oo_i$ (not necessarily definable from the
  diffeovalued field structure).  The fact that the $\tau_i$ generate
  $\tau$ implies that $R$ is co-embeddable with $\Oo_1 \cap \cdots
  \cap \Oo_m$.  This contradicts Lemma~8.29 in \cite{prdf4}.
\end{proof}

\begin{theorem}
  If $\tau$ is a $W_2$-topology on a field of characteristic 0, then
  exactly one of the following holds:
  \begin{enumerate}
  \item \label{cas1} $\tau$ is a V-topology.
  \item \label{cas2} $\tau$ is generated by two independent V-topologies.
  \item \label{cas3} $\tau$ is a DV-topology in the sense of \cite{prdf4},
    Definition~8.18.
  \end{enumerate}
  Moreover, all such topologies are $W_2$-topologies.
\end{theorem}
\begin{proof}
  Theorem~\ref{vsum:thm} and Lemma~\ref{dv:lem} show that the three
  cases are all $W_2$-topologies, and the three cases are mutually
  exclusive.  It remains to show that the three cases are exhaustive.

  Each of the three cases is closed under local equivalence.  For
  cases (\ref{cas1}-\ref{cas2}) this is by Theorem~\ref{vsum:thm}; for
  case (\ref{cas3}) this is by fiat (in Definition~8.18 of
  \cite{prdf4}).  So we may pass to a locally equivalent field.
  Therefore we may assume that $\tau$ is induced by a $W_2$-ring $R$.
  By Theorem~\ref{w2ring:thm}, one of two things happens:
  \begin{itemize}
  \item $R$ is co-embeddable with some ring of the form
    \begin{equation*}
      R' = \{x \in K : \val(x) \ge 0 \text{ and } \val(\partial x) \ge 0\}
    \end{equation*}
    induced by some dense diffeovaluation data as in \S 8 of
    \cite{prdf4}.  This ring defines the diffeovaluation topology, so
    $\tau$ is a DV-topology.
  \item There is a multi-valuation ring $S$ such that $aS \subseteq
    R$.  Then $\tau = \tau_R$ is a coarsening of $\tau_S$.  The
    topology $\tau_S$ is a $V^n$-topology (Lemma~\ref{v1:lem}).  By
    Remark~\ref{kowalsky}, $\tau$ is a $V^n$-topology.  By
    Theorem~\ref{vsum:thm} and the fact that $\tau$ is a
    $W_2$-topology, it follows that $\tau$ is generated by one or two
    independent V-topologies. \qedhere
  \end{itemize}
\end{proof}
Because \emph{non}-$V^2$-topologies are a local class, we get an
interesting corollary:
\begin{corollary}
  DV-topologies (on fields of characteristic 0) are a local class.
\end{corollary}

\begin{lemma}\label{dv2:lem}
  If $\tau$ is a DV-topology, then $\tau$ has exactly one
  V-topological coarsening.
\end{lemma}
\begin{proof}
  Otherwise, it would have two coarsenings $\tau_1, \tau_2$, by
  Theorem~\ref{defcourse}.  Let $\tau_1 + \tau_2$ denote the
  $V^2$-topology generated by $\tau_1$ and $\tau_2$, as in
  Corollary~4.3 of \cite{PZ}.  Then $\tau_1 + \tau_2$ is coarser or
  equal to $\tau$.  By Theorem~\ref{vsum:thm}, $\tau_1 + \tau_2$ is a
  $W_2$-topology but not a $W_1$-topology.
  Proposition~\ref{bump:prop} then forces $\tau = \tau_1 + \tau_2$.
  This contradicts Lemma~\ref{dv:lem}.
\end{proof}
Lemma~\ref{dv2:lem} can be used to prove that unstable fields of
dp-rank 2 admit unique definable V-topologies (Proposition~6.2 in \cite{prdf4}).

\begin{acknowledgment}
The author would like to thank 
\begin{itemize}
\item Meng Chen, for hosting the author at Fudan University, where
  this research was carried out.
\item Yatir Halevi, whose questions prompted the current paper.
\end{itemize}
{\tiny This material is based upon work supported by the National Science
Foundation under Award No. DMS-1803120.  Any opinions, findings, and
conclusions or recommendations expressed in this material are those of
the author and do not necessarily reflect the views of the National
Science Foundation.}
\end{acknowledgment}

\bibliographystyle{plain} \bibliography{mybib}{}

\end{document}